\documentclass[12pt]{article}
\usepackage{amsmath, amsthm, amssymb,enumerate,algorithm2e,tikz,caption}
\usepackage{tikz}
\usepackage{fullpage}
\usepackage{enumitem}
\usepackage{hyperref}
\usepackage{xcolor}
\usepackage{url}

\title{Isoperimetric stability in lattices}
\author{
Ben Barber\thanks{University of Manchester and Heilbronn Institute for Mathematical Research. E-mail \texttt{ben.barber@manchester.ac.uk}}
\and Joshua Erde\thanks{Graz University of Technology, Institute of Discrete Mathematics, Steyrergasse 30, 8010 Graz, Austria. E-mail: \texttt{erde@tugraz.at}}
\and Peter Keevash\thanks{Mathematical Institute, University of Oxford, Andrew Wiles Building, Radcliffe Observatory Quarter, Woodstock Road, Oxford, United Kingdom. E-mail: \texttt{\{keevash, robertsa\}@maths.ox.ac.uk}. \newline \hspace*{1.8em}Research supported
in part by ERC Consolidator Grant 647678.
} \and Alexander Roberts\footnotemark[3]}

\newtheoremstyle{case}{}{}{\normalfont}{}{shape}{:}{ }{}

\newtheorem{thm}{Theorem}[section]
\newtheorem{lem}[thm]{Lemma}

\theoremstyle{definition}

\newtheorem{claim}[thm]{Claim}

\numberwithin{equation}{section}

\newtheoremstyle{case}{}{}{\normalfont}{}{shape}{\normalfont:}{ }{}

\theoremstyle{case}

\def\comment#1{}

\newcommand{\beq}{\begin{eqnarray*}}
\newcommand{\eeq}{\end{eqnarray*}}

\def\build#1_#2^#3{\mathrel{\mathop{\kern 0pt#1}\limits_{#2}^{#3}}}

\newcommand{\bN}{\mathbb{N}}
\newcommand{\bR}{\mathbb{R}}

\newcommand{\bZ}{\mathbb{Z}}

\newcommand{\cA}{\mathcal{A}}


\numberwithin{equation}{section}

\newcommand{\sub}{\subseteq}

\newcommand{\bcl}[1]{\left\lceil #1 \right\rceil}

\newcommand{\sm}{\setminus}

\newcommand{\eps}{\varepsilon}
\newcommand{\es}{\emptyset}
\newcommand{\pl}{\partial}

\newcommand{\sd}{\bigtriangleup}

\newcommand{\ra}{\rightarrow}

\newcommand{\ova}{\overrightarrow}

\newcommand{\gG}{\gamma}
\newcommand{\dD}{\delta}

\newcommand{\kK}{\kappa}

\newcommand{\Ss}{\Sigma}

\DeclareMathOperator{\Per}{Per}

\begin{document}

\maketitle

\begin{abstract}
We obtain isoperimetric stability theorems
for general Cayley digraphs on $\bZ^d$.
For any fixed $B$ that generates $\bZ^d$ over $\bZ$, 
we characterise the approximate structure of large sets $A$
that are approximately isoperimetric in the Cayley digraph
of $B$: we show that $A$ must be close 
to a set of the form $kZ \cap \bZ^d$,
where for the vertex boundary
$Z$ is the conical hull of $B$,
and for the edge boundary
$Z$ is the zonotope generated by $B$.
\end{abstract}

\section{Introduction}

An important theme at the interface of Geometry, Analysis 
and Combinatorics is understanding the structure of
approximate minimisers to isoperimetric problems.
These problems take the general form of minimising 
surface area of sets with a fixed volume,
for various meanings of `area' and `volume'.
The usual meanings give
the Euclidean Isoperimetric Problem
considered since the ancient Greek mathematicians,
where balls are the measurable subsets of $\bR^d$ 
with a given volume which minimize the surface area.
There is a large literature on its stability, i.e.\ 
understanding the structure of approximate minimisers,
culminating in the sharp quantitative
isoperimetric inequality of
Fusco, Maggi and Pratelli \cite{FuMP08}.

In the discrete setting, isoperimetric problems
form a broad area that is widely studied within
Combinatorics (see the surveys \cite{Bez94, Leader91})
and as part of the Concentration of Measure phenomenon
(see \cite{Led05, Tal95}). Furthermore, certain particular 
settings have been intensively studied due to their applications;
for example, there has been considerable recent progress
(see \cite{KL18b,KL18a,KelLif18,VISH-MPAR})
on isoperimetric stability in the discrete cube $\{0,1\}^n$,
which is intimately connected to 
the Analysis of Boolean Functions (see \cite{OD14})
and the Kahn--Kalai Conjecture (see \cite{KK07})
on thresholds for monotone properties
(see \cite{FKNP} for the recent solution
of Talagrand's fractional version).
This paper concerns the setting of integer lattices,
which is widely studied in Additive Combinatorics, where 
the Polynomial Freiman--Ruzsa Conjecture (see \cite{G07})
predicts the structure of sets with small doubling.

For an isoperimetric problem on a digraph (directed graph) $G$,
we measure the `volume' of $A \sub V(G)$ by its size $|A|$,
and its `surface area' either by the \emph{edge boundary}
$\pl_{e,G}(A)$, which is the number of edges $\ova{xy} \in E(G)$
with $x \in A$ and $y \in V(G) \sm A$,
or by the \emph{vertex boundary} $\pl_{v,G}(A)$, 
which is the number of vertices $y \in V(G) \sm A$ 
such that $\ova{xy} \in E(G)$ for some $x \in A$.
Here we consider Cayley digraphs:\ 
given a generating set $B$ 
of $\bZ^d$, we write $G_B$ for the digraph on $\bZ^d$ with
edges $E(G_B) = \{ (u,v): v-u \in B\}$.

It is an open problem to determine the minimum possible
value of $\pl_{v,G_B}(A)$ or $\pl_{e,G_B}(A)$
for $A \sub \bZ^d$ of given size, 
let alone any structural properties 
of (approximate) minimisers;
exact results are only known 
for a few instances of $B$
(see \cite{BL87, BL91, WW77, RV12}).
It is therefore natural to seek asymptotics.
For the vertex boundary, 
this was achieved by Ruzsa \cite{R95},
who showed that for $A \sub \bZ^d$ of size $n \to \infty$,
the minimum value $\pl_{v,G_B}(A)$
is asymptotic to that achieved
by a set of the form $k C(B) \cap \bZ^d$,
where $C(B)$ is the conical hull of $B$,
i.e.\ the convex hull of $B \cup \{0\}$.
A corresponding result for the edge boundary
was obtained in \cite{BE18}:\
the minimum value $\pl_{e,G_B}(A)$
is asymptotic to that achieved
by a set of the form $kZ(B) \cap \bZ^d$, where 
$Z(B) = \{ \sum_{b \in B} bx_b: x \in [0,1]^B \}$ 
is the zonotope generated by $B$.

We will prove stability versions of both these results,
describing the approximate structure of asymptotic minimisers 
for both the vertex and edge isoperimetric problems in $G_B$.
The statements require the following notation.
Given $B \sub \bZ^d$ and $n\in \bN$ 
we write $B_n = (\kK_B(n)C(B)) \cap \bZ^d$, 
where $\kK_B(n) = \min 
 \{k \in \bN: |kC(B) \cap \bZ^d| \ge n\}$.
We also write $Z(B)=C([B])$, 
where $[B] = \{\sum B': B' \sub B\}$.
We use $\mu$ to denote Lebesgue measure.

\begin{thm}\label{thm:lattice-stability}
Let $B$ be a generating set of $\bZ^d$ with $d \ge 2$.
Then there is $K \in \bN$ so that 
for any $A \sub \bZ^d$ with $|A|=n \ge K$ and 
$\pl_{v,G_B}(A) \le d\mu(C(B))^{1/d}n^{1-1/d}(1+\eps)$,
where $Kn^{-1/2d} < \eps < K^{-1}$, 	
there exists $v \in \bZ^d$ with
$|A \sd (v+B_n)| < Kn\sqrt{\eps}$.
\end{thm}

\begin{thm}\label{thm:edge-general}
Let $B$ be a generating set of $\bZ^d$ with $d \ge 2$.
For any $\dD>0$ there are $\eps>0$ and $K\in\bN$ so that
for any $A \sub \bZ^d$ with $|A|=n \ge K$ and 
$\pl_{e,G_B}(A) \le d\mu(C([B]))^{1/d}n^{1-1/d}(1+\eps)$, 
there exists $v \in \bZ^d$ with
$|A \sd (v+[B]_n)| < \dD n$.
\end{thm}

We remark that the `square root dependence'
in Theorem \ref{thm:lattice-stability} 
constitutes a tight quantitative stability result
for the vertex isoperimetric inequality in $G_B$,
as may be seen from an example 
where $B$ consists of the corners of a cube
and $A$ is an appropriate cuboid.
We only give a qualititative statement of our
stability result in Theorem \ref{thm:edge-general} 
for the edge isoperimetric inequality in $G_B$,
as our proof does not provide good quantitative bounds.
On the other hand, we do obtain a tight 
quantitative stability result for certain $B$,
namely those for which $G_B$ is equivalent 
to the $\ell^1$-grid (see Theorem \ref{thm:edge-box} below),
thus giving a new proof of a result of
Ellis, Friedgut, Kindler and Yehudayoff \cite{EFKY}.

In proving our results, 
besides drawing on the methods of \cite{R95}
(particularly Pl\"unnecke's inequality for sumsets)
and \cite{BE18} (a probabilistic reduction to \cite{R95}),
the most significant new contribution of our paper 
is a technique for transforming discrete problems 
to a continuous setting where one 
can apply results from Geometric Measure Theory.
We will employ the sharp estimate on asymmetric index
in terms of anisotropic perimeter with respect
to any convex set $K$ due to
Figalli, Maggi and Pratelli \cite{FMP10}
(building on the case when $K$ is a ball,
established in \cite{FuMP08}).
We consider vertex isoperimetry in the next section
and then edge isoperimetry in the following section.
We conclude the paper by discussing some
potential directions for further research.

\section{Vertex isoperimetry}

This section contains the proof of our sharp
tight quantitative stability result
for the vertex isoperimetric inequality 
in general Cayley digraphs.
We start in the first subsection with a summary
of Ruzsa's approach in \cite{R95}, during which
we record some key lemmas on sumsets and fundamental
domains of lattices that we will also use in our proof.
In the second subsection we state the
Geometric Measure Theory result of \cite{FMP10}
(in a simplified setting that suffices for our purposes).
The third subsection contains a technical lemma
in elementary Real Analysis. We conclude in the final
subsection by proving Theorem \ref{thm:lattice-stability}.

\subsection{Ruzsa's approach}

The sumset of $A,B \in \bZ^d$ is defined 
by $A+B:=\{a+b:a \in A, b\in B\}$.
The vertex isoperimetric problem
in the Cayley digraph $G_B$
is equivalent to finding the minimum of $|A+B|$ 
over all sets $A$ of given size.
The following result of
Ruzsa \cite[Theorem 2]{R95} 
implies an asymptotic for this minimum.

\begin{thm}\label{vertex:extreme}
Let $B$ be a generating set of $\bZ^d$ with $d \ge 2$.
Then for any $A \sub \bZ^d$ with $|A|=n$ large we have
$|A+B| \ge d\mu(C(B))^{1/d}n^{1-1/d}(1-O(n^{-1/2d}))$.
\end{thm}

A key ingredient of Ruzsa's proof is the following
well-known inequality of Pl\"{u}nnecke \cite{P70}
(see \cite[Statement 6.2]{R95}).
We use the notation $\Ss_k(A)$ for the $k$-fold sumset of $A$
rather than the commonly used $kA$, which in this paper
denotes the dilate of $A$ by factor $k$.
 
\begin{thm}\label{stat:ruz1}
Let $k \in \bN$ and $A,B \sub \bZ^d$
with $|A|=n$ and $|A+B| = \alpha n$. 
Then there is $\es \ne A' \sub A$ with
$|A' + \Ss_k(B)| \le \alpha^k|A'|$.
\end{thm}

Given $A$ of size $n$ with $|A+B|$ small
and any $k' \in \bN$,
Theorem \ref{stat:ruz1} provides $A' \sub A$
such that $|A' + \Ss_{k'}(B)|$ is small.
The next step is to convert this into smallness
of a Minkowski sum $U+V = \{u+v: u \in U, v \in V\}$
of two bodies $U,V$ in $\bR^d$.
The final bound in Theorem \ref{vertex:extreme}
will then follow from the Brunn--Minkowski inequality
$\mu(U+V)^{1/d} \ge \mu(U)^{1/d}+\mu(V)^{1/d}$
(in the form due to Lusternik \cite{L35}).

The required $U$ and $V$ are obtained 
as $U=A'+Q$ and $V=k C(B)$, taking $k'=k+p$
with $p$ and $Q$ as in the next lemma,
where $k$ is a free parameter that can be optimised
to obtain the $O(n^{-1/2d})$ error term
in Theorem \ref{vertex:extreme}.
This lemma, which is \cite[Lemma 11.2]{R95},
provides a certain set $Q  \sub \bR^d$
with four key properties. Firstly,
it is {\em fundamental}, i.e.\ any $x \in \bR^d$
has a unique representation as $x=y+z$
with $y \in Q$ and $z \in \bZ^d$,
and so $\mu(X+Q)=|X|$ for any $X \sub \bZ^d$.
Secondly, the conclusion of the lemma
shows that $U+V$ as above is small
when $A'+\Ss_{k'}(B)$ is small.
The remaining two properties are
not explicitly stated in \cite{R95},
but are clear from the proof; we have $Q \sub Z(B)$,
and a certain technical condition,
which for brevity we will call {\em nice}:\
$Q$ is a finite union of bounded convex polytopes.
(In fact, one can take $Q$ to be a parallelepiped;
we omit the proof of this fact, as we do not use it.)

\begin{lem}\label{lem:ruz1}
Let $B$ be a generating set 
of $\bZ^d$ with $d \ge 2$ and $0 \in B$.
Then there are $p \in \bN$, $z \in \bZ^d$
and a nice fundamental set $Q \sub Z(B)$ such that
$k C(B) + Q + z \sub \Ss_{k+p}(B) + Q$ for any $k \in \bN$.
\end{lem}

\subsection{Some Geometric Measure Theory}

Here we give a brief account of the quantitative 
isoperimetric stability result 
of Figalli, Maggi and Pratelli \cite{FMP10}.
We adopt simplified definitions 
that suffice for sets that are nice,
as defined in the previous subsection
(see \cite{M08,M12} for the general setting 
of sets of finite perimeter).

For a closed convex polytope $K\subseteq \bR^d$ and a union $E$ of disjoint (possibly non-convex) closed polytopes, the perimeter of $E$ with respect to $K$ is given by
\begin{align}
\Per_K(E) = \lim_{\eps \ra 0^+}
\frac{\mu(E+\eps K)-\mu(E)}{\eps}.\label{perim0}
\end{align}
In our setting given a nice set $A$, for all $r \ge 0$, the measure of $A+rK$ and its closure $\overline{A+rK}$ are the same, that is $\mu(A+rK) = \mu(\overline{A+rK}).$ Thus for all $r\ge 0$, \eqref{perim0} gives
\begin{align}
\Per_K(\overline{A+rK}) = \lim_{\eps \ra 0^+}\frac{\mu(A+(r+\eps)K)-\mu(A+rK)}{\eps}. \label{perim}
\end{align}

The anisotropic isoperimetric problem was
posed in 1901 by  Wulff \cite{W01}, who conjectured
that minimisers of $\mbox{Per}_K$ up to null sets
are homothetic copies of $K$, giving 
$\Per_K(E) \ge d\mu(K)^{1/d}\mu(E)^{1-1/d}$. 
This was established for sets $E$ with continuous boundary 
by Dinghas \cite{D44} and for general sets $E$ 
of finite perimeter by Gromov \cite{MS86}.
It is equivalent to non-negativity
of the \emph{isoperimetric deficit} 
$\dD_K(E)$ of $E$ with respect to $K$, 
defined by
\[	\dD_K(E) := \frac{\Per_K(E)}{d\mu(K)^{1/d}\mu(E)^{1-1/d}}-1. \]
We quantify the structural similarity between $K$ and $E$
via the \emph{asymmetric index} (also known as Fraenkel asymmetry)
of $E$ with respect to $K$, which is given by
\[ 	\cA_K(E) 
= \inf\left\{\frac{\mu(E \sd (x_0 + rK))}{\mu(E)} 
: x_0 \in \bR^d \text{ and } r^d\mu(K) = \mu(E)\right\}. \]
The following is \cite[Theorem 1.1]{FMP10}.

\begin{thm}\label{shouldersofgiants}
For any $d \in \bN$ there exists $D = D(d)$
such that for any bounded convex open set $K \sub \bR^d$
and $E \sub \bR^d$ of finite perimeter we have
	\begin{align*}
		\cA_K(E) \le D\sqrt{\delta_K(E)}.
	\end{align*}
\end{thm}

\subsection{Some Real Analysis}

In this subsection we establish the following
technical lemma in elementary Real Analysis,
which will allow us to pass to the setting
of perimeters in Ruzsa's approach as described
in the first subsection, so that we can apply
the result of the second subsection. We presume the result is well-known,
but we include a proof for completeness.

\begin{lem}\label{basicanalysis}
Let $f:[a,b] \ra \bR$ be continuous
and right differentiable.
Then for any $\eps>0$ there is $x \in [a,b)$ with 
$f'_+(x) \le \tfrac{f(b)-f(a)}{b-a} + \eps$.
\end{lem}

\begin{proof}
Without loss of generality we may assume
$a=0$, $b=1$ and $f(0)=f(1)=0$.
Suppose for contradiction we have 
$f'_+(x) \ge \eps>0$ for all $x \in [0,1]$.
Let $B = \{x: f(x) \ge \eps x/2\}$.
As $f(1)=0$ we have $1 \notin B$.
As $f$ is continuous, $B$ is closed.
The required contradiction will thus follow
if we show that $B$ is open to the right,
i.e.\ for any $x \in B$ there is $\dD>0$
such that $(x,x+\dD) \sub B$.
To see this, note that for small enough $\dD$,
by definition of $f'_+(x)$ we have
$f(y) \ge f(x) + (y-x)\eps/2 \ge \eps y/2$
for any $y \in (x,x+\dD)$, so $y \in B$.
\end{proof}
\subsection{Stability}

In this final subsection we prove our theorem on
stability for vertex isoperimetry in $G_B$.

For convenience we work with sumsets,
which is an equivalent setting via the identity
$\pl_{v,G_B}(n) = |A+(B \cup \{0\})| - |A|$ 
for $A,B \subseteq \bZ^d$.

\begin{proof}[Proof of Theorem \ref{thm:lattice-stability}]
Let $B$ be a generating set of $\bZ^d$ with $d \ge 2$,
and assume without loss of generality that $0 \in B$.
Suppose $K=K(B,d)$ is sufficiently large
and $A \sub \bZ^d$ is such that $|A|=n \ge K$ 
and $|A+B| \le \alpha |A|$, where
\begin{align*}
\alpha = 1+(1+\eps)\beta n^{-1/d}, 
\ \text{ with } \
\beta =  d\mu(C(B))^{1/d}
 \ \text{ and } \
Kn^{-1/2d} < \eps < K^{-1}.
\end{align*}
We need to find $v \in \bZ^d$ with
$|A \sd (v+B_n)| < Kn\sqrt{\eps}$.

By Lemma \ref{lem:ruz1}, there are $p \in \bN$, $z \in \bZ^d$
and a nice fundamental set $Q \sub Z(B)$ such that
$k C(B) + Q + z \sub \Ss_{k+p}(B) + Q$,
where we choose $k=\bcl{n^{1/2d}}$.
By Lemma \ref{stat:ruz1}, 
there is $\es \ne A' \sub A$ with
$|A' + \Ss_{k+p}(B)| \le \alpha^{k+p}|A'|$.
It now suffices to prove the following claim.

\begin{claim}\label{c1} 
We have $|A'| \ge (1+2\eps)^{-d} n$ and 
$|A'\sd (v+B_n)| \le \tfrac12 Kn\sqrt{\eps}$
for some $v\in \bZ^d$.
\end{claim}

To see the bound on $|A'|$, 
we use the choice of $Q$
and Brunn--Minkowski to get
\begin{align*}
\alpha^{k+p}|A'| & \ge |A'+ \Ss_{k+p}(B)|
 = \mu(A'+ \Ss_{k+p}(B)+Q)
\ge  \mu(A'+k C(B) + Q) \\
&\ge \big( \mu(A'+Q)^{1/d} + \mu(k C(B))^{1/d} \big)^d
=  (|A'|^{1/d}+k\mu(C(B))^{1/d})^d.
\end{align*}
Expanding the last expression and dividing throughout by $|A'|$ then gives
\begin{align*}
1+k\beta |A'|^{-1/d} & \le \alpha^{k+p}
= 1 + (1+\eps)k\beta n^{-1/d} + O(n^{-1/d})
< 1 + (1+2\eps)k\beta n^{-1/d},
\end{align*} 
from which the first part of the claim follows.

For the second part of the claim, 
let $A_r := \overline{A'+Q+rC(B)}$ and $f(r)=\mu(A_r)$.
Since $Q$, and so $A_r$, is nice, by \eqref{perim} we have 
$\Per_{C(B)}(A_r) = f'_+(r)$ for all $r \geq 0$.

Now $f(k)-f(0)  < (\alpha^{k+p}-1)|A'| 
< (1+2\eps)k\beta n^{-1/d}|A'|$,
so by Lemma \ref{basicanalysis} with $\eps=1$
there is an $r\in[0,k)$  such that
\[
\Per_{C(B)}(A_r) \le \frac{f(k)-f(0)}{k} + 1 
< (1+3\eps) \beta n^{-1/d} |A'|.
\]
Then, by Theorem \ref{shouldersofgiants}, the asymmetric index 
$\cA_{C(B)}(A_r)$ of $A_r$ with respect to $C(B)$
is at most $D\sqrt{3\eps}$.
Thus, there is $t\in \bR^d$ such that
\[
 \mu\big( A_r \sd (t+r' C(B)) \big)
 \le D\sqrt{3\eps} \mu(A_r),
\]
where $r'  = (\mu(A_r)/\mu(C(B))^{1/d}$.
As $\mu(A_r) \leq \mu(A_k) \leq \alpha^{k+p}|A'|$, we have 
\[ r' < q := (\alpha^{k+p} |A'|/\mu(C(B))^{1/d}.\]
Since $Q \sub Z(B)$, by increasing $D$ if necessary 
we may assume that $Q \sub DC(B)$.
Let $v$ be the unique lattice point in  $t + Q$.
We will show that $v$ satisfies the claim.

To see this, we start by applying 
the triangle inequality to get
\begin{equation} \label{trineq}
|A' \sd (v + B_n)| \leq \mu((A' + Q) \sd (t + qC(B))) 
+ \mu((t + qC(B)) \sd (v + B_n + Q)).
\end{equation}
We will use the inequality
$\mu(X \sd Y) \le 2\mu(X \sm Y) + |\mu(X) - \mu(Y)|$
for any measurable $X$ and $Y$
to estimate both terms on the right of \eqref{trineq}.
Using 
\[
\mu\big( (A'+Q) \sm (t+qC(B)) \big)\le \mu\big( A_r \sm (t+r'C(B)) \big)\ \le D\sqrt{3\eps}\alpha^{k+p}|A'|,
\]
we bound the first term as
\begin{align*}
\mu\big((A' + Q) \sd (t + qC(B))\big) 
&\le 2\mu\big((A' + Q) \sm (t + qC(B))\big) 
+ |\mu\big(qC(B))\big) - |A'||
\\&\le \big(2D\sqrt{3\eps}\alpha^{k+p} + \alpha^{k+p} - 1\big)|A'|
\le 4D\sqrt\epsilon n,
\end{align*}
since $\alpha^{k+p}=1+O(n^{-1/2d})$.
For the second, we observe that
\begin{align*}
\mu\big((v + B_n + Q)\sm (t + qC(B)) \big) 
&\leq \mu\big((t + B_n + 2Q)\sm (t + qC(B)) \big)
\\ & \leq \mu\big((\kappa_B(n) + 2D)C(B))\sm qC(B) \big),
\end{align*}
so the second term 
on the right of  \eqref{trineq} is $O(n^{1-1/2d})$,
as $|B_n|$, $\mu((\kappa_B(n) + 2D)C(B))$
and $\mu(qC(B))$ are all $n + O(n^{1-1/2d})$.
This proves the claim, and so the lemma.
\end{proof}

\section{Edge isoperimetry}

In this short section we deduce 
our stability result for edge isoperimetry
from our stability result for vertex isoperimetry
proved in the previous section .
We use the reduction in \cite{BE18}.

\begin{proof}[Proof of Theorem \ref{thm:edge-general}]
Let $B$ be a generating set of $\bZ^d$ 
with $d \ge 2$. We adopt a parameter hierarchy
$n^{-1} \ll \eps \ll s^{-1} \ll \gG \ll \dD \ll |B|^{-1}$,
i.e.\ let $\dD$ be small given $|B|$,
let $\gG$ be small given $\dD$,
let $s$ be large given $\gG$,
let $\eps$ be small given $s$,
and let $n$ be large given $\eps$.
Suppose $A \sub \bZ^d$ is such that $|A|=n$ 
and $\pl_{e,G_B}(A) \le \beta n^{1-1/d}(1+\eps)$,
where $\beta = d\mu(Z)^{1/d}$
with $Z = Z(B) = C([B])$.
We need to find $v \in \bZ^d$ with
$|A \sd (v+[B]_n)| < \dD n$.

From the proof of \cite[Theorem 1]{BE18},
identifying our parameters with theirs 
via $\eta=\gG/2$ and $s=(1-\eta)t/k$, 
we have $|A+\Ss_s([B])| - |A| 
< (1+\gG) s \beta n^{1-1/d}$. As
\begin{align*}
|A+\Ss_s([B])| - |A|
= \sum_{j=0}^{s-1} |A+\Ss_{j+1}([B])| - |A+\Ss_j([B])|, 
\end{align*} 
we can therefore fix $A_+ = A+\Ss_j([B])$ 
for some $j<s$ such that
\[ |A_+ + [B]| - |A_+| < (1+\gG) \beta n^{1-1/d}
\le (1+\gG) \beta n_+^{1-1/d},\]
where $n_+=|A_+| \le n + O(n^{1-1/d})$.
By Theorem \ref{thm:lattice-stability}
we have $|A_+ \sd (v+B_{n_+})| < \tfrac12 \dD n_+$
for some $v \in \bZ^d$. Now $|A \sd (v + [B]_n)| 
< |A_+ \sd (v + [B]_{n_+})| + |A_+ \sm A|
+ |\bZ \cap [B]_{n_+}| - |\bZ \cap [B]_n| < \dD n$.
\end{proof}

Our proof of Theorem \ref{thm:edge-general} 
does not give a tight quantitative result,
but we will now demonstrate a simple trick 
that provides such a result
when the generating set $B$ 
takes the form $\{ \pm v: v \in {\cal B} \}$ 
for some integral basis ${\cal B}$ of $\bZ^d$
(which may as well be the standard basis
$\{e_1,\dots,e_d\}$).

\begin{thm}\label{thm:edge-box}
Let  $B = \{\pm e_i:i\in [d]\} \sub \bZ^d$ with $d\ge 2$.
Then there is $K \in \bN$ so that 
for any $A \sub \bZ^d$ such that $|A|=n \ge K$ and 
$\pl_{e,G_B}(A) \le 2dn^{1-1/d}(1+\eps)$,
where $Kn^{-1/2d} < \eps < K^{-1}$, 	
there exists $v \in \bZ^d$ with
$|A \sd (v+[B]_n)| < Kn\sqrt{\eps}$.		
\end{thm}

\begin{proof}
Let $A' = A + [-1/2,1/2]^d$. Then the edges of $G_B$
counted by $\pl_{e,G_B}(A)$ are in bijection
with those $(d-1)$-cubes that occur exactly once
as $x + C$ with $x \in A$ and 
$C$ a facet of $[-1/2,1/2]^d$.
Thus $\pl_{e,G_B}(A) = \Per_Z(A')$,
where $Z=Z(B)=[-1,1]^d$.
By Theorem \ref{shouldersofgiants}
we have $\cA_Z(A') \le D\sqrt{\eps}$,
i.e.\ there is $x \in \bR^d$ with
$\mu(A' \sd (x+rZ)) \le nD\sqrt{\eps}$,
where $\mu(rZ)=(2r)^d=n$.
We fix $v \in \bZ^d$ with 
$x+(r+1)Z \sub v + (r+2)Z$.
As \[(A \sm (x+(r+1)Z)) + [-1/2,1/2]^d 
\sub A' \sm (x+rZ)\] we have
$|A \sm (v+(r+2)Z)| \le nD\sqrt{\eps}$.
The theorem now follows from
\begin{align*}
A \sd (v+[B]_n)| 
= 2|A \sm (v+(r+2)Z)| + O(n^{1-1/d})
< Kn\sqrt{\eps}. 
&  \qedhere
\end{align*}
\end{proof}

We remark that by considering
$A' = A + [-1/2,1/2]^d$ as in the previous proof
one can obtain a bound for the 
edge isoperimetric problem in $G_B$
that is tight in some cases.
Indeed, by the anisotropic isoperimetric inequality,
$\pl_{e,G_B}(A) = \Per_Z(A') 
\ge d\mu(Z)^{1/d}\mu(A')^{1-1/d}
= 2d n^{1-1/d}$,
which is tight whenever 
$n=k^d$ for some $k \in \bN$.
Bollob\'as and Leader \cite{BL91}
gave a tight result for general $n$,
although they used compression techniques that
alter structure, so it is interesting that
exact results can also be obtained by geometric methods,
and moreover in two ways:\ by the argument above,
or by the Loomis-Whitney inequality as in \cite{EFKY}.

\section{Concluding Remarks}\label{waffle}

As mentioned in the introduction,
there are several challenging 
and important open problems
in isoperimetric stability,
such as the Kahn--Kalai Conjecture
and the Polynomial Freiman--Ruzsa Conjecture.
We therefore find it rather striking 
that in this short paper we have been able
to characterise isoperimetric stability
for general Cayley graphs in lattices.
Of course, the brevity of our paper masks
the fact that we have greatly relied 
on previous work, particularly an analogous
stability result of \cite{FMP10}
in Geometric Measure Theory.
This naturally suggests that further investigation
of transformations between 
the discrete and continuous settings
may be fruitful in future research.
 
For the isoperimetric problems considered in this paper,
it is natural to ask if one can obtain tighter estimates
than those in the asymptotic results of \cite{BE18,R95},
particularly for the edge boundary, where the probabilistic
reduction in \cite{BE18} introduces error terms 
that are presumably far from optimal. 
Does our improved estimate
for the edge isoperimetric inequality 
in the remark following Theorem \ref{thm:edge-box} 
hold for general $B$, i.e.\ do we always have
$\pl_{e,G_B}(A) \ge d\mu(Z(B))^{1/d} |A|^{1-1/d}$?
Do we always have
$\pl_{v,G_B}(A) \ge d\mu(C(B))^{1/d} |A|^{1-1/d}$?

It would also be interesting 
to quantify the dependence on the dimension $d$
of the constants $K$ in our theorems.
Our use of \cite{FMP10} gives $K=O(d^7)$,
whereas in the case of the $\ell^1$-grid
the authors of \cite{EFKY} show $K=O(d^{5/2})$
and conjecture $K=O(\sqrt{d})$.

Finally, a challenging direction for further research 
is to understand the structure of large sets $A$ for which
$|A+B|-|A|$ is within a multiplicative
$O(1)$ factor of its minimum value. We conjecture that
any such $A$ can be covered by $O(1)$ homothetic 
copies of $C(B)$ with total volume $O(|A|)$. 

More generally, in the spirit
of the Kahn--Kalai and Polynomial Freiman--Ruzsa Conjectures,
we pose the following (somewhat vague) question
for any (natural) isoperimetric problem:\
can any set with boundary $O(\cdot)$ of the minimum possible
be covered by $O(1)$  `canonical examples' 
of size $O(|A|)$, perhaps even with 
polynomially-related constants?

\bibliographystyle{abbrv}
\bibliography{isoperimetric}

\begin{thebibliography}{10}

\bibitem{BE18}
B.~Barber and J.~Erde.
\newblock Isoperimetry in integer lattices.
\newblock {\em Discrete Analysis}, 7, 2018.

\bibitem{Bez94}
S.~Bezrukov.
\newblock Isoperimetric problems in discrete spaces.
\newblock {\em {Extremal Problems for Finite Sets, Bolyai Soc. Math. Stud. 3}},
  pages 59--91, 1994.

\bibitem{BL87}
B.~Bollob\'{a}s and I.~Leader.
\newblock {Compressions and Isoperimetric Inequalities}.
\newblock {\em Journal of Combinatorial Theory Series A}, 56:47--62, 1991.

\bibitem{BL91}
B.~Bollob\'{a}s and I.~Leader.
\newblock Edge-isoperimetric inequalities in the grid.
\newblock {\em Combinatorica}, 11:299--314, 1991.

\bibitem{D44}
A.~Dinghas.
\newblock {{\"{U}}ber einen geometrischen Satz von Wulff f{\"{u}}r die
  Gleichgewichtsform von Kristallen}.
\newblock {\em Z. Kristallogr., Mineral. Petrogr.}, 105, 1944.

\bibitem{EFKY}
D.~Ellis, E.~Friedgut, G.~Kindler, and A.~Yehudayoff.
\newblock Geometric stability via information theory.
\newblock {\em Discrete Analysis}, 2016.

\bibitem{FMP10}
A.~Figalli, F.~Maggi, and A.~Pratelli.
\newblock A mass transportation approach to quantitative isoperimetric
  inequalities.
\newblock {\em Invent. Math.}, 182(1):167--211, 2010.

\bibitem{FKNP}
K.~Frankston, J.~Kahn, B.~Narayanan, and J.~Park.
\newblock Thresholds versus fractional expectation-thresholds.
\newblock {\em arXiv:1910.13433}, 2019.

\bibitem{FuMP08}
N.~Fusco, F.~Maggi, and A.~Pratelli.
\newblock The sharp quantitative isoperimetric inequality.
\newblock {\em Annals of Mathematics}, 168:941--980, 2008.

\bibitem{G07}
B.~J. Green.
\newblock The {P}olynomial {F}re$\breve{\i}$man-{R}uzsa {C}onjecture.
\newblock {\em Terence Tao's blog}, 2007.

\bibitem{KK07}
J.~Kahn and G.~Kalai.
\newblock Thresholds and expectation thresholds.
\newblock {\em Combinatorics, Probability and Computing}, 16:495--502, 2007.

\bibitem{KL18b}
P.~Keevash and E.~Long.
\newblock Stability for vertex isoperimetry in the cube.
\newblock {\em arXiv:1807.09618}, 2018.

\bibitem{KL18a}
P.~Keevash and E.~Long.
\newblock A stability result for the cube edge isoperimetric inequality.
\newblock {\em J. Combin. Theory Ser. A}, 155:360--375, 2018.

\bibitem{KelLif18}
N.~Keller and N.~Lifshitz.
\newblock Approximation of biased {B}oolean functions of small total influence
  by {DNF}'s.
\newblock {\em Bulletin of the London Mathematical Society}, 50(4):667--679,
  2018.

\bibitem{Leader91}
I.~Leader.
\newblock Discrete isoperimetric inequalities.
\newblock {\em {Probabilistic Combinatorics and its Applications}}, 1991.

\bibitem{Led05}
M.~Ledoux.
\newblock The concentration of measure phenomenon.
\newblock {\em Mathematical Surveys and Monographs}, 89, 2005.

\bibitem{L35}
L.~Lusternik.
\newblock {Die Brunn-Minkowskische Ungleichung f{\"{u}}r beliebige messbare
  Mengen}.
\newblock {\em C. R. Acad. Sci. URSS}, 8:55--58, 1935.

\bibitem{M08}
F.~Maggi.
\newblock Some methods for studying stability in isoperimetric type problems.
\newblock {\em Bull. Amer. Math. Soc.}, 45:367--408, 2008.

\bibitem{M12}
F.~Maggi.
\newblock {\em {Sets of finite perimeter and geometric variational problems: an
  introduction to Geometric Measure Theory}}.
\newblock Number 135 in Cambridge Studies in Advanced Mathematics. Cambridge
  University Press, 2012.

\bibitem{MS86}
V.~Milman and G.~Schechtman.
\newblock {\em {Asymptotic Theory of Finite-dimensional Normed Spaces. With an
  appendix by M. Gromov}}, volume 1200 of {\em Lecture Notes in Mathematics}.
\newblock Springer, Berlin, 1986.

\bibitem{OD14}
R.~O'Donnell.
\newblock {\em Analysis of {B}oolean functions}.
\newblock Cambridge University Press, 2014.

\bibitem{P70}
H.~Pl{\"{u}}nnecke.
\newblock {Eine zahlentheoretische Anwendung der Graphentheorie}.
\newblock {\em J. Reine Angew. Math}, 243:171--183, 1970.

\bibitem{VISH-MPAR}
M.~Przykucki and A.~Roberts.
\newblock Vertex-isoperimetric stability in the hypercube.
\newblock {\em Journal of Combinatorial Theory Series A}, 172, 2020.

\bibitem{RV12}
J.~Radcliffe and E.~Veomett.
\newblock Vertex isoperimetric inequalities for a family of graphs on
  {$\mathbb{Z}^k$}.
\newblock {\em Electron. J. Combin.}, 19(2):P45, 2012.

\bibitem{R95}
I.~Ruzsa.
\newblock Sets of sums and commutative graphs.
\newblock {\em Studia Sci. Math. Hungar}, 30:127--148, 1995.

\bibitem{Tal95}
M.~Talagrand.
\newblock Concentration of measure and isoperimetric inequalities in product
  spaces.
\newblock {\em {Publ. Math. I.H.E.S.}}, 81:73--203, 1995.

\bibitem{WW77}
D.~Wang and P.~Wang.
\newblock Discrete isoperimetric problems.
\newblock {\em SIAM J. Appl. Math.}, 32:860--870, 1977.

\bibitem{W01}
G.~Wulff.
\newblock {Zur Frage der Geschwindigkeit des Wachsrurms und der Aufl{\"{o}}sung
  der Kristallfl{\"{a}}chen}.
\newblock {\em Z. Kristallogr.}, 34:449--530, 1901.

\end{thebibliography}
\end{document}